\RequirePackage{fix-cm}
\documentclass[12pt,a4paper,oneside,english]{scrbook}
\usepackage{graphicx, amsmath, amssymb, amsthm, algorithmicx, algorithm}
\usepackage{algpseudocode}
\newtheorem{definition}{Definition}
\newtheorem{lemma}{Lemma}

\newtheorem{proposition}{Proposition}
\newcommand{\sign}{\operatorname{sign}}
\begin{document}

\title{$\mathcal O(n)$ working precision inverses for symmetric tridiagonal Toeplitz matrices with $\mathcal O(1)$ floating point calculations. 
}


\author{Manuel Radons         
}




\maketitle

\textbf{Abstract: }
\textit{
A well known numerical task is the inversion of large symmetric tridiagonal Toeplitz matrices, i.e., matrices whose entries equal $a$ on the diagonal and $b$ on the extra diagonals ($a, b\in \mathbb R$).  
The inverses of such matrices are dense and there exist well known explicit formulas by which they can be calculated in $\mathcal O(n^2)$. In this note we present a simplification of the problem that has proven to be rather useful in everyday practice: If $\vert a\vert > 2\vert b\vert$, that is, if the matrix is strictly diagonally dominant, its inverse is a band matrix to working precision and the bandwidth is independent of $n$ for sufficiently large $n$. Employing this observation, we construct a linear time algorithm for an explicit tridiagonal inversion that only uses $\mathcal O(1)$ floating point operations. On the basis of this simplified inversion algorithm we outline the cornerstones for an efficient parallelizable approximative equation solver.}

\textbf{Keywords}\ \textit{Tridiagonal, symmetric,  Toeplitz, Explicit inverse, Working precision, Linear time, Equation solving}

\section{Introduction}
The inversion of symmetric tridiagonal Toeplitz matrices is a rather prominent task in applied mathematics. In general, the inverse of such matrices is dense and there exist explicit formulas, such as those presented in \cite{RefFons}, that allow performance of the inversion in $\mathcal O(n^2)$. However, for large $n$ ($\ge 100,000$), which one encounters frequently in numerical applications, this quadric complexity is too costly. 

In the second section of this note we derive a simplified inversion formula for strictly diagonally dominant matrices of the (symmetric tridiagonal Toeplitz) form \eqref{symmTri} and simplify it further in section three. We then use this simplified form to devise a Matlab style pseudocode for an explicit inversion with $\mathcal O(n)$ read/write operations and integer additions (in the form of running indices) that uses only $\mathcal O(1)$ floating point operations.
In the final section we use the simplified inverse structure to outline an approximative equation solver that runs in roughly $4n$ fused multiply-adds and is communication free and thus parallelizable in a straightforward fashion.

\section{First Simplification of Explicit Inversion}\label{intro}

It is shown in \cite[Cor. 4.2.]{RefFons} that the inverse of an $n\times n$ tridiagonal matrix of the form
\begin{align}\label{symmTri}
	T\ =\ \begin{bmatrix}
	    	a & b &&& \\
	    	b & a & b &&& \\
	    	 &\ddots &\ddots &\ddots & \\
	    	 &&b&a&b \\
	    	 &&&b&a
	    \end{bmatrix}.
\end{align}
 is given by the formula
\begin{align*}
\left(T^{-1}\right)_{ij}\ =\  \left\{\begin{array}{ll} (-1)^{i+j}\frac 1b \frac{U_{i-1}(a/2b)U_{n-j}(a/2b)}{U_n(a/2b)} & \text{if}\ i\le j \\ \\  (-1)^{i+j}\frac 1b \frac{U_{j-1}(a/2b)U_{n-i}(a/2b)}{U_n(a/2b)}& \text{if}\ i>j  \end{array}\right. ,
\end{align*}
with
\begin{align}\label{U_n_x}
U_n(x)\ =\ \frac{r_+^{n+1} - r_-^{n+1}}{r_+ - r_-},
\end{align}
where
\begin{align*}
r_\pm\ =\ x\pm \sqrt{x^2-1}
\end{align*}
are the two solutions of the quadratic equation $r^2 - 2xr + 1=0$. Let $\vert x\vert>1$, i.e., $\vert a\vert > 2\vert b\vert$. Then the solutions $r_\pm$ are real and we have
\begin{align*}
r_+r_-\ =\ \left(x+\sqrt{x^2-1}\right)\left(x-\sqrt{x^2-1}\right)\ =\ x^2 -\left(\sqrt{x^2-1}\right)^2\ =\ 1.
\end{align*}
Consequently,
\begin{align}\label{substMinusLog}
r_+\ =\ \frac{1}{r_-},\qquad \text{which implies}\qquad \log\vert r_+\vert \ =\ - \log\vert r_-\vert
\end{align}
for arbitrary bases of the logarithm. Moreover, we either have 
\begin{align*}
\vert r_+\vert\ >\ 1 \qquad \text{or}\qquad \vert r_-\vert\ >\ 1 \ . 
\end{align*}
We assume, without loss of generality, the first case, i.e., $\vert r_+\vert > 1$. Moreover, we define 
\begin{align}\label{subst2log}
\phi \ := \ \log_2 \vert r_+\vert\ >\ 0 \qquad \text{and}\qquad \tau\ := \ b(r_+ -r_-)\ .
\end{align}

Let $\delta\in \mathbb N_{>0}$. We call an arithmetic whose smallest representable magnitude is $2^{-\delta}$ a $\delta$-arithmetic. In this sense, for example, single precision using subnormal representation is a $149$-arithmetic (see, e.g., \cite{RefFloat}).
\begin{definition}
We say $s,t\in\mathbb R$ are numerically equal in $\delta$-arithmetic, if $\vert s-t\vert < 2^{-\delta}$.
\end{definition}

In the following we assume $i\le j$. Due to the symmetry of $T$ -- and thus of $T^{-1}$ -- analogous results follow for the case $i>j$. Given said assumption it holds

\begin{align}\label{subst1}
\left(T^{-1}\right)_{ij}  \overset{\eqref{U_n_x}}=&\ (-1)^{i+j}\ \frac 1\tau \  \frac{(r_+^i -r_-^i)(r_+^{n-j+1} -r_-^{n-j+1})}{r_+^{n+1} -r_-^{n+1}}\\ 
\overset{\eqref{substMinusLog}}=&\ (-1)^{i+j}\ \frac 1\tau \  \frac{r_+^{n+i-j+1} -r_+^{n-i-j+1} -r_+^{-n+i+j-1} +r_+^{-n-i+j-1}}{r_+^{n+1} -r_+^{-n-1}}\ .\label{invSimple}
\end{align}
Assume $i\ge \left\lceil \frac n2 \right\rceil$. We investigate the difference between the exact value \eqref{invSimple} of  $\left(T^{-1}\right)_{ij}$ and the simplified term 
\begin{align}\label{simpleTerm}
(-1)^{i+j}\ \frac 1\tau \ \ \left( r_+^{i-j}  -r_+^{-2(n+1)+i+j} \right)
\end{align}
for large $n$. 
\begin{lemma}\label{iGeOneHalf}
Let $i\ge \left\lceil \frac n2\right\rceil$ and $10\le k\in \mathbb N$ and define 
\begin{align}
N \ := \ \left\lceil \frac{k}{\phi}\right\rceil\ .
\end{align}
Then for all $n\ge N$ it holds 
\begin{align}\label{lemmaStatement}
\left\vert \left(T^{-1}\right)_{ij}\ -\ (-1)^{i+j}\ \frac 1\tau \ \ \left( r_+^{i-j}  -r_+^{-2(n+1)+i+j} \right) \right\vert \ < \ 3\left\vert \frac{1}{\tau 2^k}\right\vert \ .
\end{align}
\end{lemma}
\begin{proof} The left hand term in \eqref{lemmaStatement} equals:
\begin{align*}
&\ \left\vert \left(T^{-1}\right)_{ij}\ -\ (-1)^{i+j}\ \frac 1\tau \ \ \frac{r_+^{n+i-j+1}  -r_+^{-n+i+j-1} }{r_+^{n+1}} \right\vert \\
=&\ \left\vert \left(T^{-1}\right)_{ij}\ -\ \left[\left(T^{-1}\right)_{ij} \frac{r_+^{n+1} -r_+^{-n-1}}{r_+^{n+1}} \ - \ (-1)^{i+j}\ \frac{r_+^{-2(n+1)-i+j}-r_+^{-i-j}}{\tau}\right] \right\vert \\
=&\ \left\vert  \left(T^{-1}\right)_{ij} r_+^{-2(n+1)} \ - \ (-1)^{i+j}\ \frac{r_+^{-2(n+1)-i+j}-r_+^{-i-j}}{\tau} \right\vert\\
\le &\ \left\vert \frac{r_+^{i-j} - r_+^{-i-j} - r_+^{-2(n+1)+i+j} +r_+^{-2(n+1) -i+j}}{\tau (r_+^{2(n+1)}-1)}\right\vert\ +\ \left\vert\frac {r_+^{-2(n+1)-i+j} - r_+^{-i-j}}{\tau}\right\vert \\
=&:\ \xi_1\ +\ \xi_2 \ . 
\end{align*}
Noting that $2^{2\phi(n+1)}>0$, for all $n \ge N$ it holds
\begin{align} \label{kBound}
2^{1.9\phi n}\ <\ 2^{2\phi(n+1)}-1 \qquad \text{and}\qquad  2^k\le 2^{\phi n}\ .
\end{align}
Keeping in mind that by hypothesis $i\ge \left\lceil \frac n2\right\rceil$, we have
\begin{align}
\xi_1 \ <& \ \frac{1}{\vert\tau\vert}\left\vert \frac{r_+^{i-j} - r_+^{-i-j} - r_+^{-2(n+1)+i+j} +r_+^{-2(n+1) -i+j}}{ r_+^{1.9n}}\right\vert\nonumber \\
<& \ \frac{1}{\vert\tau\vert}\left\vert \frac{r_+^{0} - r_+^{-n} - r_+^{-2} +r_+^{-1.5n}}{r_+^{1.9n}}\right\vert\nonumber \\ 
< & \ 4 \left\vert \frac{ r^{-1.9n}}{\tau}\right\vert \ <\ \left\vert \frac{ r^{-n}}{\tau}\right\vert \ .\label{xi1}
\end{align}
Moreover,  
\begin{align}\label{xi2}
\xi_2\ \overset{i\ge \left\lceil \frac n2\right\rceil}< \ \left\vert\frac {r_+^{-1.5n} - r_+^{-n}}{\tau}\right\vert \ < \
2 \left\vert \frac{ r^{-n}}{\tau}\right\vert \ .
\end{align}
Together, \eqref{xi1} and \eqref{xi2} yield
\begin{align*}
\xi_1 \ + \ \xi_2 < 3 \left\vert \frac{ r^{-n}}{\tau}\right\vert \ \le \ 3\left\vert \frac{1}{\tau 2^k}\right\vert \  .
\end{align*}
This completes the proof.
\end{proof}
An analogous statement immediately follows for $i < \left\lceil \frac n2\right\rceil$.
\begin{lemma}\label{iSmallerOneHalf}
Let $i < \left\lceil \frac n2\right\rceil$ and $10\le k\in \mathbb N$ and define 
\begin{align}
N \ := \ \left\lceil \frac{k}{\phi}\right\rceil\ .
\end{align}
Then for all $n\ge N$ it holds 
\begin{align}
\left\vert \left(T^{-1}\right)_{ij}\ -\ (-1)^{i+j}\ \frac 1\tau \ \ \left( r_+^{i-j}  -r_+^{-i-j} \right) \right\vert \ < \ 3\left\vert \frac{1}{\tau 2^k}\right\vert \ .
\end{align}
\end{lemma}
\begin{proof}
This follows from Lemma \ref{iGeOneHalf} and the persymmetry of $T$ and thus of $T^{-1}$.
\end{proof}
 
Combining Lemma \ref{iGeOneHalf} and \ref{iSmallerOneHalf}, we get:
\begin{proposition}
Assuming $i\le j$, let  
\begin{align*}
k_\delta\ \ge \ \max (10, \delta+\log_2 \vert\tau\vert -\log_2 3)\qquad \text{and} \qquad N_\delta \ := \ \left\lceil \frac{k_\delta}{\phi}\right\rceil\ .
\end{align*}
Moreover, define 
\begin{align*}
\left(T_\delta^{-1}\right)_{ij}\ =\  \left\{\begin{array}{ll} (-1)^{i+j}\ \frac 1\tau \ \ \left( r_+^{i-j}  -r_+^{-i-j} \right) & \text{if}\ i< \left\lceil \frac n2\right\rceil \\ \\  (-1)^{i+j}\ \frac 1\tau \ \ \left( r_+^{i-j}  -r_+^{-2(n+1)+i+j} \right) & \text{if}\ i\ge \left\lceil \frac n2\right\rceil  \end{array}\right. .
\end{align*}
Then, for all $n\ge N_\delta$, $\left(T_\delta^{-1}\right)_{ij}$ and $\left(T^{-1}\right)_{ij}$ are numerically equal in $\delta$-arithmetic.
\end{proposition}

\section{Further Simplification and Algorithm}

Due to the bisymmetry of $T$ -- and thus of $T^{-1}$ -- it suffices to calculate the entries of one of the four sections of the subdivision of $T^{-1}$ which is induced by the main- and the counter-diagonal. All other entries can be derived using the appropriate reflections. Our attention will thus be restricted to the case $n\ge N_\delta$, $i\le j$ and $i+j\le n+1$, which especially means that $i\le (n+1)/2$.

The first choice in the design of an actual algorithm for the inverse calculation is to decide whether to first perform the subtraction of the terms in the brackets and then multiply the result with $(-1)^{i+j}\frac 1\tau$ or vice versa. We investigate the options: 

At first note that the value of the $r_+^{i-j}$ and thus of the $(-1)^{i+j}\frac 1\tau r_+^{i-j}$ exclusively depends on the horizontal distance of the corresponding entry from the main diagonal. They numerically equal zero if either
\begin{align*}
i-j \ < \ - \frac{\delta}{\phi} \quad \text{or}\quad i-j \ < \ - \frac{\delta}{\phi}\ +\ \frac{\log_2\vert \tau\vert}{\phi}\ =: -\alpha_\delta \ , 
\end{align*}
respectively. Recalling that by assumption $\vert r_+\vert > 1$, we have 

\begin{align}\label{ge}
\vert r_+^{i-j}\vert > \vert r_+^{-i-j}\vert\ .
\end{align}
Hence, if $r_+^{i-j}\ / \ \frac 1\tau r_+^{i-j}$ is numerically zero, then so is $r_+^{-i-j}\ / \ \frac 1\tau r_+^{-i-j}$ -- which implies that for both possible approaches the $T_\delta^{-1}$ are band matrices with constant bandwidth, say $b_{r, \delta}$ [which is necessarily $\le \max(\frac \delta\phi, \alpha_\delta)$], for all $n\ge N_\delta$.

Moreover, since $i\le j$ and thus $2i\le i+j$, the terms $r_+^{-i-j}$ and $(-1)^{i+j}\frac 1\tau r_+^{-i-j}$ numerically equal zero if either 
\begin{align*}
-2i \ < \ - \frac{\delta}{\phi} \qquad \text{or}\qquad -2i \ < \ - \alpha_\delta \ .
\end{align*}
This implies that for all $n\ge \max (N_\delta, \frac \delta\phi, \alpha_\delta)$ the number of rows for which the $(T^{-1}_t)_{ij}$ do \textit{not} simply equal $(-1)^{i+j}\frac 1\tau r_+^{i-j}$ is a \textit{constant}, say $c_{r, \delta}$.
The following pseudocode, which we will refine in the sequel, already establishes the statement of this article's title: that the explicit inverse can be calculated, up to working precision, in $\mathcal{O}(n)$, using only $\mathcal{O}(1)$ floating point operations.

\begin{algorithm}
\caption{ Outline explicit tridiagonal inversion}

\begin{algorithmic}[1]
 
\State $T_t^{-1}=zeros(n,n)$

\For {$i=1:c_{r, \delta}$}
\For {$j=i:i+b_{r, \delta}-1$}
\State $T_t^{-1}(i,j)= (-1)^{i+j}\frac 1\tau(r_+^{i-j}-r_+^{-i-j})$
\EndFor
\EndFor

\State OtherValues = zeros$(b_{r, \delta},1)$

\For {$i=1:b_{r, \delta}$}

\State OtherValues $(i) = (-1)^{1+i}\frac 1\tau r_+^{1-i}$

\EndFor

\For {$i=c_{r, \delta}+1: \left\lceil \frac n2 \right\rceil$}

\State $T_t^{-1}(i,i: i+b_{r, \delta}-1)$=OtherValues$(:)$

\EndFor

\State Use symmetries of $T_t^{-1}$ to fill up the rest of the matrix. 

\State \textbf{return} $T_t^{-1}$

\end{algorithmic}

\end{algorithm}


As a first step towards the refinement of the algorithm outline we note that if $\vert x\vert >1$, then $\vert x \vert > \vert \sqrt{x^2-1}\vert$. Which implies  
\begin{align*}
\operatorname{sign}\left(r_+\right)\ = \ \operatorname{sign}\left( x\right) \ , 
\end{align*}
where $\operatorname{sign}$ denotes the signum function.

We now distinguish the two cases $x<0$, which is the case if and only if $a$ and $b$ have different signs, and $x>0$. 

$x\ <\ 0:$ Then  
\begin{align}\label{absRplus}
(-1)^{i+j}r_+^{i-j}\ =\ (-1)^{i+j}(-1)^{i-j} \vert r_+\vert^{i-j} \ = \ (-1)^{2i}\vert r_+\vert^{i-j} \ = \ \vert r_+\vert^{i-j} \ .
\end{align}
Recalling \eqref{ge},  
\begin{align*}
\operatorname{sign}\left(r_+^{i-j}  -r_+^{-i-j}\right)\ = \ \operatorname{sign}\left(r_+^{i-j}  \right) \ .
\end{align*}
Hence, $T^{-1}$ is either nonnegative or nonpositive, depending on the sign of $\tau$. 

Moreover, $i-j$ is odd if and only if $i+j\ (=i-j+2j)$ is odd. Thus it holds 
\begin{align*}
\operatorname{sign}\left(r_+^{i-j}\right)\ = \ \operatorname{sign}\left( r_+^{-i-j}\right) \ , 
\end{align*}
which, by \eqref{absRplus}, implies $(-1)^{i+j}r_+^{-i-j} =  \vert r_+\vert^{-i-j}$ and thus
\begin{align*}
(-1)^{i+j}\ \frac 1\tau \  \left( r_+^{i-j}  -r_+^{-i-j} \right)\ = \ \frac 1\tau  \ \left( \vert r_+\vert^{i-j}  -\vert r_+\vert^{-i-j} \right)  \ .
\end{align*}

$x\ >\ 0:$ Then  
\begin{align*}
(-1)^{i+j}\ \frac 1\tau  \ \left( r_+^{i-j}  -r_+^{-i-j} \right)\ = \ (-1)^{i+j}\frac 1\tau  \ \left( \vert r_+\vert^{i-j}  -\vert r_+\vert^{-i-j} \right)  \ ,
\end{align*}
which implies that the diagonals of $T^{-1}$ have alternating signs. The relevance of this observation lies in the fact that it allows us to assign the signs diagonal-wise and thus circumvent the effort of the sign calculations.

\begin{algorithm}
\caption{ Refined explicit tridiagonal inversion}

\begin{algorithmic}[1]

 \Procedure{EXPLICIT INVERSE}{a, b, n , $\delta$}
 \Comment{$\delta$ is a precision parameter.} 
 \State $r_+= a/2b$
 \State $r=-r_+$
 \Comment{The sign switch allows to omit computation of the powers of $-1$.} 
 \State $b_{r, \delta}=\alpha_\delta$
 \State $Values=zeros(1,b_{r,\delta})$
 \State $Values(1)=1/\tau$
 \For {$i=2:b_{r, \delta}$}
 \State $Values(i)=Values(i-1)/r$
 \EndFor
 
 \State $T_\delta^{-1}=zeros(n,n)$ 
 \Comment{Inefficient. Will be stored differently in real life applications.}

\For {$i=1:\left\lceil \frac n2\right\rceil$}
\State $T_\delta^{-1}(i, i:i+b_{r, \delta}-1)=Values(:)$
\EndFor

\State $c_{r, \delta}=b_{r, \delta}-2$

\For{$i=1:c_{r, \delta}$} 
\For{$j=i:b_{r, \delta}-1$}
\State $T_\delta^{-1}(i,j)=T_\delta^{-1}(i,j)-Values(i+j)$
\EndFor
\EndFor

\For{$i=\left\lceil \frac n2\right\rceil+1:n-b_{r, \delta}$}\Comment{1st use of persymmetry.}

\For{$j=i:i+b_{r, \delta}-1$}
\State $T_\delta^{-1}(i,j)=T_\delta^{-1}(n-j+1,n-i+1)$
\EndFor
\EndFor

\For{$i=n-b_{r, \delta}+1:n$}\Comment{2nd use of persymmetry.}
\For{$j=i:n$}\Comment{2nd loop instead of evaluation $\max(n, i+b_{r, \delta}-1)$.}
\State $T_\delta^{-1}(i,j)=T_\delta^{-1}(n-j+1,n-i+1)$
\EndFor
\EndFor

\For{$i=2:b_{r, \delta}$}\Comment{1st use of symmetry.}
\For{$j=1:i-1$}
\State $T_\delta^{-1}(i,j)=T_\delta^{-1}(j,i)$
\EndFor
\EndFor

\For{$i=b_{r, \delta}+1$}\Comment{2nd use of symmetry.}
\For{$j=i-b_{r, \delta}+1:i-1$} \Comment{2nd loop instead of evaluation $\min(1, i-b_{r, \delta}+1)$.}
\State $T_\delta^{-1}(i,j)=T_\delta^{-1}(j,i)$
\EndFor
\EndFor

\State \textbf{return} $T_t^{-1}$
\EndProcedure
\end{algorithmic}

\end{algorithm}

We may now answer the initial question, whether we should first subtract or multiply. For a single term $\frac 1\tau  \left( r_+^{i-j}  -r_+^{-i-j} \right)$ it saves a multiplication to perform the subtraction of $r_+^{i-j}$ and $r_+^{-i-j}$ first. But, defining $b_{r, \delta}:=\alpha_\delta$, the set $$B:= \left \{\frac 1\tau r_+^{1-j}: 1\le j\le b_{r, \delta}\right \}$$ contains all terms $\frac 1\tau r_+^{i-j}$ and $\frac 1\tau r_+^{-i-j}$ that do not equal zero numerically. So it suffices to perform $b_{r, \delta}$ floating point multiplications to compute the values in $B$. 

As for the subtractions: Since it holds $(i-j)-(-i-j)=2i \ge 2$, let $c_{r, \delta}:=  b_{r, \delta} -2$. The number of tuples $(i,j)$ with $i, j\ge 1$ and $i\le j$ such that $i+j\le c_{r, \delta}$ -- which is also the number of $\frac 1\tau r_+{-i-j}$ that do not equal zero numerically and thus the number of floating point subtractions to be performed -- equals
$(c_{r, \delta}-1)+ (c_{r, \delta}-2) + \dots + 1= c_{r, \delta}(c_{r, \delta}-1)/2$ tuples $(i, j)$ such that $i+j \le l$. 

Hence, performing the multiplications first and then the additions, we merely need to perform $\alpha_\delta$ floating point additions and $(\alpha_\delta-2)(\alpha_\delta-3)/2$ floating point multiplications.

These observations are summarized in the refined pseudocode. Note that the algorithm is optimized for a minimal operations count, not memory efficiency. In real life situations one would certainly \textit{not} store the inverse in an $n\times n$ array. However, the building blocks of the algorithm can be extracted and plugged into a memory efficient implementation. 

For this we note that $A:= T^{-1}_\delta$ can be decomposed into three blocks $B,B'$ and $C$ as follows: Let $w$ be the number of rows in $A$, where one or more nonzero entries are a sum ($w$ is even due to the bisymmetry of $A$) and define $$d\ :=\ \max\left(\frac w2,b_{r,\delta}\right)\; .$$  
Then designate the first and last $d$ rows of $A$ as $B$ and $B'$, respectively. The block in between is $C$. Clearly, $B'$ equals $B$ rotated by $180^\circ$. It thus suffices to store one of these blocks, which costs $\mathcal O(1)$ memory (more precisely $\mathcal O(\delta)$ -- but this is the same for all practical purposes). Let $d<i<n-d$. Then shifting the nonzero entries of the $i$-th row of $A$ by one to the right yields row $i+1$. Hence it suffices to store the nonzero enries of a single row of $C$, which gives an overall memory requirement of $\mathcal O(1)$.

Moreover, a feature of the algorithm which has proven to be immensely practical is that said values have to be computed (and stored) \textit{only once} and can then be used to construct the inverse for \textit{any} sufficiently large $n$.

\section{Equation solving (outline) and outlook}

Define $A,B, B', C, d, r$ and the bandwidth $b_{r,\delta}=:\alpha$ as in the last section. We want to (approximatively) solve a linear system of the form $$Tx\ =\ b\ \Leftrightarrow \ x\ =\ T^{-1}b\ \approx\ T^{-1}_\delta b\ =\ Ab\; .$$ A naive but inefficient approach would be to simply multiply $b$ with $A$ (inefficient, since the bandwidth $b_{r,\delta}$ of $A$ is likely larger than the constants in current algorithms for the solution of tridiagonal systems such as those presented in \cite{RefHockney} and \cite{RefMcNally}). 

We propose another procedure: Decompose $x$ into $x^{(1)}:= Bb, x^{(2)}:= B'b$ and $x^{(3)}:= Cb$. The vectors $x^{(1)}$ and $x^{(2)}$ can be computed in $\mathcal O(\delta)\approx \mathcal O(1)$. Let $d<i<n-d$. Then 
$$x^{(3)}_{i-d}\ =\ e_i^TAb\ =\ \sum_{j=i-\alpha}^{i+\alpha}a_{ij}b_j\ =\ \sum_{j=i-\alpha}^ia_{ij}b_j\ + \sum_{j=i+1}^{i+\alpha}a_{ij}b_j\ =:\ \varphi_i + \varphi_i' $$ and it can easily be verified that
\begin{align*}
x^{(3)}_{i-d+1}\ &=\ \sum_{j=i-\alpha+1}^{i+1}a_{ij}b_j\ + \sum_{j=i+2}^{i+\alpha+1}a_{ij}b_j \\
&=\ \frac{\varphi_i-a_{i,i-\alpha}\cdot b_{i-\alpha}}{r}\ +\  r\cdot \varphi_i'\ +\ a_{i+1,i+\alpha+1}\cdot b_{i+\alpha+1} \ .
\end{align*}
This row-to-row update can be refined into a repeatedly applicable update scheme. For this we remark that the block $C$ is Toeplitz. Now choose $d<i<n-d$ arbitrarily and define 
$$b'\ := a_{i,i-\alpha}\cdot b\qquad \text{and}\qquad b''\ := a_{ii}\cdot b\ .$$  
These rescalings cost $2n$ multiplications. (Albeit, since we are scaling vectors, these can likely be implemented more efficiently than $2n$ arbitrary multiplications.) Now define
$$\varphi_{i+1}\ :=\ \frac{\varphi_i-b'_{i-\alpha}}{r} + b''_{i+1}$$ and 
$$\varphi'_{i+1}\ :=\ r\cdot \varphi'_i -b''_{i+1} + b_{i+1+\alpha}\ .$$
Performing the updates in this manner costs $2n$ additions plus $2n$ fused multiply-adds plus another $n$ additions of the $\varphi_i$ and $\varphi_i'$. One might argue (this has to be investigated further) that the subtraction of $b'_{i-\alpha}$ in the first equality can be omitted due to its negligible numerical impact. This would spare $n$ subtractions and reduce the overall cost of the so-performed update scheme to roughly $4n + \mathcal O(1)$ multiplications and additions, which is about in line with the computational cost of the approximative tridiagonal Toeplitz solver presented in \cite{RefMcNally} -- albeit the aforementioned algorithm by McNally et. al. does not require symmetry.

The proposed update scheme has the following advantages and disatvantages:
\begin{itemize}
\item Pro: Easy to implement.
\item Pro: Fast.
\item Pro: Communication free if block $C$ is decomposed further and thus parallelizable.
\item Con: Possibly unstable. 
\end{itemize}

To give definitive statements about the practical usefulness of the above update scheme, these aspects have to be be studied further. Another possibly productive line of investigation is the extension of the statements of this note to non-symmetric Toeplitz matrices. We will follow up on these questions.



\begin{thebibliography}{}
%
%
\bibitem{RefFons}
C.M. da Fonseca and J. Petronilho, Explicit inverses of some tridiagonal matrices, Linear Algebra and its Applications, 325, 7-21 (2001)

\bibitem{RefHockney}
R.W. Hockney, A fast direct solution of Poissons equation using
Fourier analysis, Journal of the ACM, 12(1), 95-113 (1965)

\bibitem{RefMcNally}
    J.M. McNally, L.E. Garey and R.E. Shaw, A communication-less
parallel algorithm for tridiagonal Toeplitz systems, Journal of Computational and
Applied Mathematics, 212, 260 - 271 (2008) 

\bibitem{RefFloat}
IEEE Task P754, ANSI\slash IEEE 754-1985, Standard for Binary
                 Floating-Point Arithmetic,  pub-IEEE-STD, 1985

\end{thebibliography}


\end{document}